\documentclass[11pt]{article}
\usepackage{indentfirst}
\usepackage{mathrsfs}
\usepackage{latexsym,lineno}
\usepackage{algorithm}
\usepackage{enumerate}
\usepackage{algorithmic}
\usepackage{epsfig}
\usepackage{color}
\usepackage{amsmath}\usepackage{fleqn}\usepackage{verbatim}
\usepackage{epsf}
\usepackage{amsthm}\usepackage{graphicx, float}\usepackage{graphicx}
\usepackage{amsfonts}
\usepackage{amssymb}\usepackage{graphpap}
\usepackage{epic}\usepackage{curves}
\usepackage{tikz}
\usepackage{subfigure}
\usepackage{mathrsfs}
\usepackage{pstricks}
\usepackage{hyperref}

\newtheorem{claim}{Claim}

\newtheorem{theorem}{Theorem}[section]

\newtheorem{corollary}{Corollary}
\newtheorem{lemma}{Lemma}[section]
\newtheorem{conjecture}{Conjecture}[section]
\numberwithin{equation}{section}

\title{\bf Spectral extremal graphs of  planar graphs\\ with fixed size\thanks {Research was partially supported by the National
		Nature Science Foundation of China (grant number
		12331012}}
\date{}
\author {Liangdong Fan$^{1}$, \,  Liying Kang$^{1,2}$\thanks{\em Corresponding author. Email address: lykang@shu.edu.cn (L. Kang)}, \, Jiadong Wu$^{1}$\\
	{\small $^{1}$ Department of Mathematics, Shanghai University,
		Shanghai 200444, P.R. China}\\
	{\small$^{2}$Newtouch Center for Mathematics of Shanghai University,
		Shanghai,  China, 200444}}

\begin{document}
	
	\maketitle
	
	\begin{abstract}
Tait and Tobin [J. Combin. Theory Ser. B 126 (2017) 137--161] determined the unique spectral extremal graph over all outerplanar graphs and the unique spectral extremal graph  over all planar graphs when the number of vertices is sufficiently large. In this paper we consider the spectral extremal problems of outerplanar  graphs and planar graphs with fixed number of edges. We prove that the outerplanar graph on $m \geq 64$ edges with the maximum spectral radius is $S_m$, where $S_m$ is a star with $m$ edges. For planar graphs with $m$ edges,	our main result shows that the spectral extremal graph is $K_2 \vee \frac{m-1}{2} K_1$ when $m$ is odd and sufficiently large, and $K_1 \vee (S_{\frac{m-2}{2}} \cup K_1)$ when $m$ is even and sufficiently large. Additionally, we obtain  spectral extremal graphs for path, cycle and matching
in outerplanar graphs and spectral extremal graphs for path, cycle and complete graph on $4$ vertices in planar graphs.
		
\bigskip
		
		\medskip
		
		\noindent{\bf AMS (2000) subject classification:}  05C35; 05C50
	\end{abstract}
	
	\section{Introduction}
	Let  $\mathcal {F}$ be a  family of graphs. A graph $G$ is called {\em $\mathcal {F}$-free } if for every $F\in \mathcal {F}$, there is no subgraph of $G$ isomorphic to $F$. If $\mathcal {F}=\{F\}$, then $G$ is called  $F$-free. The {\em Tur\'an number} $ex(n, F)$ is the maximum number of edges in a graph on $n$ vertices that is  $\mathcal {F}$-free. Determining Tur\'an number of graphs is one of the central problem in extremal combinatorics.

	The {\em spectral radius} of a graph $G$ is the spectral radius of its adjacent matrix $A(G)$, denoted by $\rho(G)$. By the Perron-Frobenius theorem, the spectral radius $\rho(G)$ is the largest eigenvalue of $A(G)$.
	Nikiforov \cite{Niki2002cpc} proposed a spectral Tur\'an problem which asks to
	determine the maximum spectral radius of an $\mathcal {F}$-free graph with n vertices. This can be viewed as
	the spectral analogue of Tur\'an type problem.
	The spectral  Tur\'an problem has received a great deal of
	attention in the past decades (\cite{Guiduli1996, {LLF2022}, {Niki2007},  {Nikiforov2010}, {Niki2010}, {Niki2011}, {TAIT2017}, {ZL2023}, {WKX2023}, {Wilf1986}}).
	In this paper we consider spectral Tur\'an problem for graphs with fixed number of edges.

	The problem of characterizing graphs of given size with maximal spectral radii was initially
	posed by Brualdi and Hoffman \cite{Brualdi85} as a conjecture, and was completely solved by Rowlinson \cite{Rowlin}. Furthermore, Rowlinson \cite{Rowlin1} determined the unique spectral extremal graph over all Hamiltonian graphs with fixed number of edges.
	Nosal \cite{Nosal1970} showed that for any triangle-free graph $G$ with size $m$, the spectral radius $\rho(G)$ satisfies $\rho(G) \leq \sqrt{m}$. Lin, Ning and
	Wu [7] slightly improved the bound to  $\rho(G) \leq \sqrt{m-1}$ when G is non-bipartite and triangle-free.
	Nikiforov \cite{Niki2002cpc, Niki2006-walks} extended  Nosal's result to $K_{r+1}$-free graphs. In 2009,
	Nikiforov \cite{Niki2009laa}
	obtained a sharp upper bound
	for  $C_4$.
	In 2021, Zhai, Lin and Shu \cite{ZLS2021} determined the unique spectral extremal graph for
	$K_{2, r+1}$ with given size.
	\begin{theorem}[\cite{ZLS2021}]\label{le1}
		Let  $ r \geq 2 $ and $ m \geq 16r^2 $. If $G$ is a $K_{2, r+1}$-free graph with size $m$,
		then $ \rho(G) \leq \sqrt{m} $, and equality holds if and only if $ G $ is a star.
	\end{theorem}
	For the $(r+1)$-book $B_{r+1}$, Zhai, Lin and Shu \cite{ZLS2021}  conjectured that $\rho(G)\leq \sqrt{m}$ for $r\geq 2, m$ sufficiently large and every
	$B_{r+1}$-free graph $G$ with size $m$, with equality if and only if $G$ is complete bipartite. This conjecture was solved by Nikiforov \cite{Niki2021}. Recently, Li, Zhai and Shu \cite{LZS2024}
	gave an upper bound of  $\rho(G)$ for
	$C_k^+$-free graph $G$ with size $m$, where $C_k^{+}$
	is a graph on $k$ vertices obtained from $C_k$
	by adding a chord between two vertices with distance two.

	Researchers  also have  shown a strong interest in studying the spectral radius of planar graphs. The spectral radius of planar graphs is particularly useful in geography as a measure of the overall connectivity of these graphs \cite{Cvetkovic1990}. In 1990, Cvetković and Rowlinson \cite{Cvetkovic1990} proposed the following conjecture.
	
	\begin{conjecture}[\cite{Cvetkovic1990}]\label{outerplanar conjecture}
		The outerplanar graph on $n$ vertices with the maximum spectral radius is $K_1 \vee P_{n-1}$.
	\end{conjecture}
	
	Shortly thereafter, Boots and Royle \cite{Boots1991}, and independently Cao and Vince \cite{CAO1993}, proposed the following conjecture.
	
	\begin{conjecture}[\cite{Boots1991}\cite{CAO1993}]\label{planar conjecture}
		The planar graph on $n \geq 9$ vertices with the maximum spectral radius is $P_2 \vee P_{n-2}$.
	\end{conjecture}

	In 2017, Tait and Tobin \cite{TAIT2017} proved  these conjectures  when $n$ is sufficiently large.
	In this paper, we focus on the spectral extremal problems of  outerplanar and planar graphs with given size.
	
	In Theorem \ref{le1}, let \( r = 2 \). This leads to the conclusion that the spectral extremal graph 	on \( m \geq 64 \) edges
	for  \( K_{2,3} \)
   is  $S_m$. Given that a star graph is also \( K_4 \)-minor-free and \( K_{2,3} \)-minor-free, we deduce the spectral extremal graph  of outerplanar graphs as follows.
	\begin{theorem}\label{th1}
		Let $ G $ be a outerplanar graph with size $ m \geq 64$, then $ \rho(G) \leq \sqrt{m} $, and equality holds if and only if $ G $ is a star.
	\end{theorem}
	
	For a  graph $F$, let $\mathrm{SPEX}_{\mathcal{OP}}(m, F)$ ($\mathrm{SPEX}_{\mathcal{P}}(m, F)$) denote the set of graphs attaining the maximum spectral radius among all $F$-free outerplanar (planar) graphs with $m$ edges and no isolated vertices.

	If a graph $F$ is not a subgraph of a star graph $S_m$, then by Theorem \ref{th1},
	$S_m$ is the unique graph in
	$\mathrm{SPEX}_{\mathcal{OP}}(m, F)$. Since $ P_k \ (k \geq 4 ),  C_l\ (l \geq 3) $, and $ tK_2\ (t \geq 2) $ are not subgraphs of $ S_m $, we have the following corollary.
	
	\begin{corollary}
		For $ m \geq 64 $, $ k \geq 4 $, $ l \geq 3 $, and $ t \geq 2 $,
		\begin{align*}
			\mathrm{SPEX}_{\mathcal{OP}}(m, P_k) &= \mathrm{SPEX}_{\mathcal{OP}}(m, C_l) = \mathrm{SPEX}_{\mathcal{OP}}(m, tK_2) = \{S_m\}.
		\end{align*}
	\end{corollary}
	
	Let $G_m$ be a graph defined as follows. If $m$ is odd, $G_m = K_2 \vee \frac{m-1}{2} K_1$; if $m$ is even, $G_m = K_1 \vee \left(S_{\frac{m-2}{2}} \cup K_1\right)$.
	\begin{theorem}\label{th3}
		Let $G$ be a planar graph with $m$ edges. Suppose $m$ is sufficiently large, then $\rho(G) \leq \rho(G_m)$, with equality holding if and only if $G \cong G_m$. \end{theorem}
	
	Similarly, if a graph $ F $ is not a subgraph of $ G_m $, then by Theorem \ref{th3},  $G_m$ is the unique graph in $\mathrm{SPEX}_{\mathcal{P}}(m, F)$. Since  cycles $ C_l \ (l \geq 5),  K_4 $, and $tK_2 \ (t \geq 3)$ are not subgraphs of $G_m $, we get the following corollary.
	
	\begin{corollary}
		For sufficiently large $m$,  $ l \geq 5 $, and $ t \geq 2$,
		\begin{align*}
			\mathrm{SPEX}_{\mathcal{P}}(m, C_l) &= \mathrm{SPEX}_{\mathcal{P}}(m, K_4) = \mathrm{SPEX}_{\mathcal{P}}(m, tK_2) = \{G_m\}.
		\end{align*}
	\end{corollary}
	
	This paper is organized as follows.
	In Section 2, we introduce some preliminary definitions  and lemmas.
	Sections 3 is dedicated to the proof of Theorem \ref{th3}.
	We conclude this paper with some remarks and a conjecture in the last section.

	\section{Preliminaries}
	In this section, we introduce some notations and  necessary lemmas  which will be used later.
	
	Consider a simple graph $G = (V(G), E(G))$, where $V(G)$ and $E(G)$ represent the vertex set and edge set of $G$, respectively. The number of edges is denoted by $e(G) = |E(G)|$.
	For any vertex $v \in V(G)$, let $N_i(v)$ denote the set of vertices in $G$ at distance $i$ from $v$. Specifically, $N(v) = N_1(v)$ denotes the neighbors of $v$, and $d(v) = |N(v)|$ is the degree of $v$ in $G$. For a subset $S \subseteq V(G)$, $N_S(v)$ refers to the set of neighbors of $v$ within $S$.
	When considering two disjoint subsets $S, T \subseteq V(G)$, $G[S]$ denotes the subgraph induced by $S$, while $G[S, T]$ represents the bipartite subgraph with vertex set $S \cup T$ that includes all edges between $S$ and $T$ in $G$. The number of edges within $S$ is given by $e(S) = |E(G[S])|$, and the number of edges between $S$ and $T$ is $e(S, T) = |E(G[S, T])|$.
	For two graphs $G$ and $H$, the {\em disjoint union} of $G$ and $H$ is denoted by $G\cup H$. The \textit{join} of $G$ and $H$, denoted by $G\vee H$, is the graph obtained from $G \cup H$ by adding all possible edges between $G$ and $H$.
	
	A {\em planar graph} $G$ is a graph that can be drawn on the plane such that the edges of $G$ intersect only at their endpoints. An {\em outerplanar graph} is a planar graph that can be drawn such that all vertices lie on the outer face of the drawing.
	It is well known that a graph  $G$ is planar if and only if it does not contain a $K_5$-minor or $K_{3, 3}$-minor. A graph $G$ is outerplanar if and only if it does not contain a $K_4$-minor or   $K_{3, 2}$-minor.
	For any planar graph $ G $, it is known that\begin{align}\label{pmbds}
		e(S) \leq 3|S| - 6 \quad \text{and} \quad e(S, T) \leq 2(|S| + |T|) - 4,
	\end{align}
	where $ S $ and $ T $ are disjoint subsets of $ V(G) $.

	Let $\mathbf{x}=(x_1,\cdots,x_n)^{\mathrm{T}}$ be an eigenvector corresponding to $\rho(G)$.
	For a vertex $v\in V(G)$, we will use $x_v$ to denote the eigenvector entry of $v$ corresponding to $\mathbf{x}$. With this notation, for any $u\in V(G)$,
	the eigenvector equation becomes
	\begin{align}\label{q1}
		\rho(G)x_u=\sum \limits_{uv\in E(G)}x_v.
	\end{align}
	Furthermore, we also have
	\begin{align}\label{q2}
		\rho(G)^2x_u=d_G(u)x_u + \sum_{v \in N(u)} \sum_{w \in N(v) \setminus \{u\}} x_w .
	\end{align}
	Let $G$ be a graph on $n$ vertices with $m$ edges.  It is well-known that
	\begin{align}\label{rhosx}
		\frac{2m}n\leq\rho(G)\leq\sqrt{2m}.
	\end{align}
	Based on the relationship between the spectral radius and the eigenvector, we can easily derive the following lemma.
	\begin{lemma}\label{sxl}
		Let $ G $ and $ H $ be two graphs with $ V(G) = V(H) $, $ \boldsymbol{x} $ and $ \boldsymbol{y} $ be the eigenvectors corresponding to the spectral radius of  $G$ and $H$, respectively. Then
		$$
		\boldsymbol{x}^\mathrm{T} A(G) \boldsymbol{y} = \sum_{ij \in E(G)} (x_i y_j + x_j y_i),
		$$
		and
		$$
		\boldsymbol{x}^\mathrm{T} \boldsymbol{y} (\rho(H) - \rho(G)) = \boldsymbol{x}^\mathrm{T} (A(H) - A(G)) \boldsymbol{y}.
		$$
	\end{lemma}

	It is known that $A(G)$ is irreducible nonnegative for a connected graph $G$. From the Perron–
	Frobenius Theorem, there is a unique positive unit eigenvector corresponding to $\rho (G)$, which is
	called the {\sl Perron vector} of G.
	
	The following lemmas will be used in our proof.
	\begin{lemma}[\cite{WXH2005}] \label{yb}
		Let $G$ be a connected graph
		and $(x_1, \ldots , x_n)^{\mathrm{T}}$ be a Perron vector of $G$,
		where the coordinate $x_i$ corresponds to the vertex $v_i$.
		Assume that
		$v_i, v_j \in V(G)$ are vertices such that $x_i \ge x_j$, and $S\subseteq N_G(v_j) \setminus N_G(v_i)$ is non-empty.
		Denote $G'=G- \{v_jv : v\in S\} +
		\{v_iv : v\in S\}$. Then $\rho (G) < \rho (G')$.
	\end{lemma}
	\begin{lemma}[\cite{Li1979}]\label{subgraph}
		Let $G$ be a connected graph. If $G'$ is a proper subgraph of $G$, then $\rho(G')< \rho(G)$.
	\end{lemma}
	
	\section{Proof of Theorem \ref{th3}}

	Throughout this section, we always assume that $G$ is the planar graph with maximum spectral radius on $m$ edges and no isolated vertices.

	The sketch of our proof is as follows: A lower bound on $\rho(G)$ is given by the conjectured extremal example. Using the leading eigenvector of $A(G)$, we deduce that there are two vertices $u', u''$ whose entries in the leading eigenvector are close to $1$. Next we show  that the entries of  leading eigenvector corresponding to vertices in $V\setminus \{u', u''\}$ are very small and the degrees of vertices in $V\setminus \{u', u''\}$ are at most $2$. Based on these properties, we refine the structure of $G$. Finally, we show that it must be the conjectured graph.

	\begin{lemma}
		$G$ is connected.
	\end{lemma}
	
	\begin{proof}
		Suppose to the contrary that $G$ is not connected. Assume
		$G_1,\ldots,G_s$ are the components of $G$ and $\rho(G_1)=\max\{\rho(G_i)|\ i\in [s]\}$, then $\rho(G)=\rho(G_1)$ and
		$|V(G_1)|\leq n-2$. Since $G$ has no isolated vertices, we arbitrarily choose two vertices $u\in V(G_1), v\in V(G_2)$, add an edge $uv$ and delete an edge in $G_2$ and isolated vertices. The obtained graph is denoted by $G'$. Then $G'$ is also a planar graph with size $m$, and $G_1\cup \{uv\}$ is a connected subgraph of $G'$. So $\rho(G')\geq \rho(G_1\cup\{uv\})>\rho(G_1)=\rho(G)$ by Lemma \ref{subgraph}, which is a contradiction.
		Therefore, $G$ is connected.
	\end{proof}

	Let
	$\boldsymbol{x} = (x_{1}, x_{2}, \cdots, x_{n})^{\mathrm{T}}$ be the positive eigenvector corresponding to $\rho(G)$, normalized such that $\max_{v \in V(G)} x_v = 1$. Denote by $u' \in V(G)$ the vertex with  maximum eigenvector entry,
	i.e., $x_{u'} = 1$. For any real number $\varepsilon \leq 1$, we define $L^\varepsilon = \left\{u \in V(G) \mid x_u \geq \varepsilon\right\}.$
	\begin{lemma}\label{rhox}
		$\rho(G) >\sqrt{m}$.
	\end{lemma}
	\begin{proof}Since $ G_m$ is a planar graph with $m$ edges, we have $ \rho(G) \geq \rho(G_m)$.
		Simple calculations show that when $m$ is odd and sufficiently large,
		\begin{align*}
			\rho(G_m) = \rho\left( K_2 \vee \frac{m-1}{2} K_1 \right) = \frac{1 + \sqrt{4m - 3}}{2} > \sqrt{m}.
		\end{align*}
		
		When \( m \) is even and sufficiently large, since the planar graph \( K_2 \vee \frac{m-2}{2} K_1 \) is a proper subgraph of \( K_1 \vee \left( S_{\frac{m-2}{2}} \cup K_1 \right) \), we have
		\begin{align*}
			\rho(G_m) > \rho\left( K_2 \vee \frac{m-2}{2} K_1 \right) = \frac{1 + \sqrt{4m - 7}}{2} > \sqrt{m}.
		\end{align*}
		In conclusion, $\rho(G) > \sqrt{m}$.
	\end{proof}

	\begin{lemma}\label{lr}
		$|L^\varepsilon| < \frac{2\sqrt{m}}{\varepsilon}.$
	\end{lemma}
	\begin{proof}
		From Lemma \ref{rhox} and  (\ref{q1}), for each $u \in L^\varepsilon$, we  derive
		\begin{align}\label{eq2}
			\sqrt{m}\varepsilon < \rho(G) x_u = \sum_{v \in N_G(u)} x_v \leq d_G(u).
		\end{align}
		Summing inequality (\ref{eq2}) over all $u \in L^\varepsilon$, we obtain
		\begin{align*}
			\sqrt{m}\varepsilon \cdot |L^\varepsilon| < \sum_{u \in L^\varepsilon} d_G(u) \leq \sum_{u \in V(G)} d_G(u) \leq 2m,
		\end{align*}
		which implies that $|L^\varepsilon| < \frac{2\sqrt{m}}{\varepsilon}$.
	\end{proof}
	
	\begin{lemma}\label{l1}
		$|L^\frac{1}{1000}| \leq 20000$.
	\end{lemma}
	\begin{proof}
		Let $u$ be any vertex of $G$. For convenience, we denote $L_i^\varepsilon(u) = N_i(u) \cap L^\varepsilon$ and $\overline{L_i^\varepsilon}(u) = N_i(u) \setminus L^\varepsilon$.
		By Lemma \ref{rhox} and  (\ref{q2}), we have
		\begin{align}
			mx_u &< \rho^2(G) x_u = d_G(u)x_u + \sum_{v \in N_1(u)} \sum_{w \in N_1(v) \setminus \{u\}} x_w \nonumber\\[2mm]
			&\leq d_G(u)x_u + \sum_{v \in N_1(u)} \sum_{w \in L_1^\varepsilon(u) \cup L_2^\varepsilon(u)} x_w + \sum_{v \in N_1(u)} \sum_{w \in \overline{L_1^\varepsilon}(u) \cup \overline{L_2^\varepsilon}(u)} x_w, \label{eq6}
		\end{align}
		where the last inequality follows from the fact that $$N_1(v) \setminus \{u\} \subseteq N_1(u) \cup N_2(u) = L_1^\varepsilon(u) \cup L_2^\varepsilon(u) \cup \overline{L_1^\varepsilon}(u) \cup \overline{L_2^\varepsilon}(u).$$
		
		Note that $N_1(u) = L_1^\varepsilon(u) \cup \overline{L_1^\varepsilon}(u)$.  We  obtain
		\begin{align}
			\sum_{v \in N_1(u)} \sum_{w \in L_1^\varepsilon(u) \cup L_2^\varepsilon(u)} x_w &= \sum_{v \in L_1^\varepsilon(u)} \sum_{w \in L_1^\varepsilon(u) \cup L_2^\varepsilon(u)} x_w + \sum_{v \in \overline{L_1^\varepsilon}(u)} \sum_{w \in L_1^\varepsilon(u) \cup L_2^\varepsilon(u)} x_w \nonumber\\[2mm]
			&\leq 2e(L_1^\varepsilon(u)) + e(L_1^\varepsilon(u), L_2^\varepsilon(u)) + \sum_{v \in \overline{L_1^\varepsilon}(u)} \sum_{w \in L_1^\varepsilon(u) \cup L_2^\varepsilon(u)} x_w.\label{eq7}
		\end{align}
		
		Recall that $L_1^\varepsilon(u) \cup L_2^\varepsilon(u) \subseteq L^\varepsilon$. According   Lemma \ref{lr} and  (\ref{pmbds}), we obtain
		\begin{align}
			2e(L_1^\varepsilon(u)) + e(L_1^\varepsilon(u), L_2^\varepsilon(u)) &\leq 2(3|L_1^\varepsilon(u)|-6) + (2(|L_1^\varepsilon(u)| + |L_2^\varepsilon(u)|)-4) \nonumber\\[2mm]
			&< 8|L^\varepsilon| < \frac{16\sqrt{m} }{\varepsilon}.\label{eq8}
		\end{align}
		
		Additionally, note that for each $w \in \overline{L_1^\varepsilon}(u) \cup \overline{L_2^\varepsilon}(u)$, it holds that $x_w < \varepsilon$. Then
		\begin{align}
			\sum_{v \in N_{1}(u)}\sum_{w \in \overline{L_{1}^{\varepsilon}}(u) \cup \overline{L_{2}^{\varepsilon}}(u)} x_{w} &= \sum_{v \in L_{1}^{\varepsilon}(u) \cup \overline{L_{1}^{\varepsilon}}(u)}\sum_{w \in \overline{L_{1}^{\varepsilon}}(u) \cup \overline{L_{2}^{\varepsilon}}(u)} x_{w} \nonumber \\[2mm]
			&< \left(e(L_{1}^{\varepsilon}(u), \overline{L_{1}^{\varepsilon}}(u) \cup \overline{L_{2}^{\varepsilon}}(u)) + 2e(\overline{L_{1}^{\varepsilon}}(u)) + e(\overline{L_{1}^{\varepsilon}}(u), \overline{L_{2}^{\varepsilon}}(u))\right) \cdot \varepsilon \nonumber \\[2mm]
			&\leq 2e(G)\varepsilon = 2m\varepsilon \label{eq9}
		\end{align}
		
		Combining (\ref{eq6})–(\ref{eq9}), we obtain
		\begin{align}
			mx_{u} &< d_{G}(u)x_{u} + \sum_{v \in \overline{L_{1}^{\varepsilon}}(u)}\sum_{w \in L_{1}^{\varepsilon}(u) \cup L_{2}^{\varepsilon}(u)} x_{w} + \left(\frac{16\sqrt{m} }{\varepsilon} + 2m\varepsilon\right).
			\label{eq10}
		\end{align}
		
		Now we will prove that for any vertex $u \in L^{\frac{1}{1000}}$, we have $d_G(u) \geq \frac{m}{10000}$. Suppose to the contrary  that there exists a vertex $z \in L^{\frac{1}{1000}}$ such that $d_G(z) < \frac{m}{10000}$.
		Substituting $u = z$, $\varepsilon=\xi = \sqrt{10}m^{-\frac{1}{4}}$ into
		(\ref{eq10}), we derive
		\begin{align}\label{eq11}
			\frac{m}{1000} < d_G(z)x_{z} + \sum_{v \in \overline{L_1^{\xi}}(z)} \sum_{w \in L_1^{\xi}(z) \cup L_2^{\xi}(z)} x_w + \frac{16m^{\frac{3}{4}}}{\sqrt{10}} + 2\sqrt{10} m^{\frac{3}{4}}.
		\end{align}
		
		Moreover, it can be deduced that
		\begin{align*}
			d_{G}(z)x_{z} + \sum_{v \in \overline{L_{1}^{\xi}}(z)}\sum_{w \in L_{1}^{\xi}(z) \cup L_{2}^{\xi}(z)} x_{w} &\leq d_{G}(z) + e(\overline{L_{1}^{\xi}}(z), L_{1}^{\xi}(z) \cup L_{2}^{\xi}(z)) \nonumber \\
			&< d_{G}(z) + 2(|\overline{L_{1}^{\xi}}(z)| + |L_{1}^{\xi}(z)| + |L_{2}^{\xi}(z)|) \nonumber \\[2mm]
			&\leq 3d_{G}(z) + 2|L^{\xi}|,
		\end{align*}
		where the second inequality follows from (\ref{pmbds}). Since $|N_1(z)| = d_G(z) < \frac{m}{10000}$, $\left|L^{\xi}\right| < \frac{2\sqrt{m}}{\xi}$, and $m$ is  sufficiently large, combining with (\ref{eq11}) we have
		\begin{align*}
			\frac{m}{1000}  &< 3d_{G}(z) + 2|L^{\xi}| + \frac{16m^{\frac{3}{4}}}{\sqrt{10}} + 2\sqrt{10} \cdot m^{\frac{3}{4}} \nonumber \\[2mm]
			&< \frac{3m}{10000} + \frac{4m^{\frac{3}{4}}}{\sqrt{10}} + \frac{16m^{\frac{3}{4}}}{\sqrt{10}} + 2\sqrt{10} \cdot m^{\frac{3}{4}} < \frac{m}{1000},
		\end{align*}
		which is a contradiction. Therefore, for any vertex  $u \in L^\frac{1}{1000}$, we have $d_G(u) \geq \frac{m}{10000}$. Summing this inequality over all $u \in L^\frac{1}{1000}$, we obtain
		\begin{align*}
			|L^\frac{1}{1000}| \cdot \frac{m}{10000} \leq \sum_{u \in L^\frac{1}{1000}} d_G(u) \leq 2e(G) = 2m,
		\end{align*}
		thus $|L^\frac{1}{1000}| \leq 20000$, and the result follows.\end{proof}
	
	For simplicity, we denote $L^\frac{1}{1000}$, $N_i(u) \cap L^\frac{1}{1000}$, and $N_i(u) \setminus L^\frac{1}{1000}$ by  $L$, $L_i(u)$, and $\overline{L_i(u)}$, respectively.
	
	\begin{lemma}\label{du}
		For every $u \in L$, it holds that $ d_G(u) > \left(\frac{x_u}{2} - 0.002\right)m $.
	\end{lemma}
	
	\begin{proof}
		Let $ S$ be a subset of $ \overline{L_1}(u) $ where each vertex has at least two neighbors in $ L_1(u) \cup L_2(u) $. We claim that $ |S| \leq |L_1(u) \cup L_2(u)|^2 $. If $ |L_1(u) \cup L_2(u)| = 1 $, then $ S = \emptyset $, as expected. Now, assume $ |L_1(u) \cup L_2(u)| \geq 2 $. Suppose instead that $ |S| > |L_1(u) \cup L_2(u)|^2 $. Given that vertices in $ S $ have only $ \binom{|L_1(u) \cup L_2(u)|}{2} $ choices for selecting two neighbors from $ L_1(u) \cup L_2(u) $, we can find two vertices in $ L_1(u) \cup L_2(u) $ that share at least $ \lceil \frac{|S|}{\binom{|L_1(u) \cup L_2(u)|}{2}} \rceil \geq 3 $ common neighbors in $ S $. Moreover, since $ u \notin L_1(u) \cup L_2(u) $ and $ S \subseteq \overline{L_1}(u) \subseteq N_1(u) $, there exists a copy of $ K_{3, 3} $ in  $ G $ , which is impossible as $ G $ is a planar graph. Therefore, $ |S| \leq |L_1(u) \cup L_2(u)|^2 $. Then
		\begin{align}
			e(\overline{L_1}(u), L_1(u) \cup L_2(u)) &= e(\overline{L_1}(u) \setminus S, L_1(u) \cup L_2(u)) + e(S, L_1(u) \cup L_2(u)) \nonumber \\
			&\leq |\overline{L_1}(u) \setminus S| + |L_1(u) \cup L_2(u)| \cdot |S| \nonumber \\
			&\leq d_G(u) + (20000)^3 \nonumber \\
			&< d_G(u) + 0.001m, \label{eq12}
		\end{align}
		where the second to last inequality follows from $\overline{L_1}(u) \subseteq N_1(u)$ and $ |L_1(u) \cup L_2(u)| \leq |L| \leq 20000 $ (by Lemma \ref{l1}), and the last inequality holds as $m$ is sufficiently large. Substituting $\varepsilon = \frac{1}{1000}$ into (\ref{eq10}), and combining  with (\ref{eq12}), we  derive
		\begin{align*}
			mx_u < d_G(u) + \left(d_G(u) + 0.001m\right) + 0.003m.
		\end{align*}
		Thus
		\begin{align*}
			d_G(u) > \frac{mx_u}{2} - \frac{0.001m + 0.003m}{2} = \left(\frac{x_u}{2} - 0.002\right)m.
		\end{align*}
	\end{proof}
	\begin{lemma}\label{3.6}
		$ d_G(u') \leq 0.505m $.
	\end{lemma}
	\begin{proof}
		By contradiction, assume that $ d_G(u') > 0.505m $. Let $ A $ be the neighborhood of $ u' $, and  $ B = V(G) \setminus (A \cup \{u'\}) $.
		Write $A=\{u_1, u_2, \ldots, u_{|A|}\}, B=\{v_1, v_2, \ldots, v_{|B|}\}$.
		We have the following claim.
		
		\begin{claim}\label{claim1}
			For any $ u, v \in G \setminus \{u'\}$,	$ x_{u} + x_{v} < 1 $.
		\end{claim}
		
		\begin{proof}
			By the definition of $ L $, we may assume $ u $ and $ v $ both belong to $ L $. Noting that $ e(G) = m $. The condition $ d_G(u') > 0.505m $ implies $ d_G(u) + d_G(v) < 0.496m $. Furthermore, by Lemma \ref{du}, we have
			\begin{align*}
				\left(\frac{x_{u}+x_{v}}{2} - 0.004\right)m < d_G(u) + d_G(v) < 0.496m,
			\end{align*}
			which leads to $ x_{u} + x_{v} < 1 $.
		\end{proof}
		Next we will prove that $ u' $ is incident to
		every edge in $G$. Otherwise, there exists an edge in $ E(A)\cup E(A, B)\cup E(B) $.
		Set $|E(A)|=s, |E(A)|+|E(A, B)|=t, |E(A)|+|E(A, B)|+|E(B)|=l$.
		We delete all edges in $ E(A)\cup E(A, B)\cup E(B) $. For each removed edge, we add a new  vertex $ w_i $ and
		an edge connecting vertex $ w_i $ to $ u' $, thereby forming the new graph $ G' $.
		Clearly, $ G' $ is a disjoint  union of a star with size $m$ and $|B|$ isolated vertices.
		Let $ W=\{w_1, w_2, \dots, w_l\} $ be  the set of added  vertices, where $w_i \ (1\leq i\leq s), w_i \ (s+1\leq i\leq t), w_i\ (t+1\leq i\leq l)$ are the added vertices corresponding to the deleted edges in $E(A), E(A, B)$ and $E(B)$ respectively.
		For the graph $ G' $, let $$
		\boldsymbol{y} = \begin{pmatrix}
			1, y_{u_1}, \dots, y_{u_{|A|}}, y_{v_1}, \dots, y_{v_{|B|}}, y_{w_1}, \dots, y_{w_{l}}
		\end{pmatrix}^{\mathrm{T}}
		$$ be
		the non-negative eigenvector corresponding to the spectral radius $ \rho(G')$
		 with maximum entry $1$.
		 Then $y_{u'}=1$,  and $y_{v_i}=0$ for any $ 1\leq i\leq |B| $.
		 By symmetry, for any $1\leq i\leq |A|, 1\leq j\leq l$, we have $ y_{u_i} = y_{w_j}$.

		Let $ G^* = G \cup I_{|W|} $. Then $ \rho(G^*) = \rho(G) $ and $ |V(G^*)| = |V(G')| $.
		Recall that
		$\boldsymbol{x} = (x_{1}, x_{2}, \cdots, x_{n})^{\mathrm{T}}$ is a positive eigenvector corresponding to $ \rho(G)$ and  $x_{u'} = 1$.
		Then $$\boldsymbol{z}
		= (1, x_{u_1}, \dots, x_{u_{|A|}}, x_{v_1}, \dots, x_{v_{|B|}}, z_{w_1},  \dots, z_{w_l}),$$ where $z_{w_i}=0\ (i=1, \ldots, l)$,  is a  non-negative eigenvector corresponding to $\rho( G^*)$.

		According to Lemma \ref{sxl}, it follows that
		\begin{align}\label{eq1}
			\rho(G') - \rho(G^*)  = \frac{2}{\boldsymbol{z}^{\mathrm{T}} \boldsymbol{y}} \left(\mathcal{E}_A+\mathcal{E}_{AB}+\mathcal{E}_B\right),
		\end{align}
		where
		\begin{align*}
			\mathcal{E}_A&=\sum\limits_{\substack{i=1}}^{s} \left( x_{u'} y_{w_i} + z_{w_i} y_{u'}\right) - \sum_{ u_j u_k \in E(A)} \left(x_{u_j} y_{u_k} + x_{u_k} y_{u_j} \right),\\
			\mathcal{E}_{AB}&=\sum\limits_{\substack{i=s+1}}^{t} 	\left( x_{u'} y_{w_i} + z_{w_i} y_{u'}\right) -\sum_{u_j v_k \in E(A, B) }	\left(x_{u_j} y_{v_k} + x_{v_k} y_{u_j} \right),\\
			\mathcal{E}_{B}&=\sum\limits_{\substack{i=t+1}}^{l} \left( x_{u'} y_{w_i} + z_{w_i} y_{u'}\right) - \sum_{v_j v_k \in E(B)}	\left(x_{v_j} y_{v_k} + x_{v_k} y_{v_j} \right).
		\end{align*}

		By Claim \ref{claim1}, we have
		\begin{align*}
			\mathcal{E}_A&=\sum\limits_{\substack{i=1}}^{s} \left( x_{u'} y_{w_i} + z_{w_i} y_{u'}\right) - \sum_{ u_j u_k \in E(A)} \left(x_{u_j} y_{u_k} + x_{u_k} y_{u_j} \right)\\
			&=\sum\limits_{\substack{i=1}}^{s} \left( y_{w_i} + 0\right) - \sum\limits_{\substack{i=1}}^{s} \left(x_{u_j} y_{w_i} + x_{u_k} y_{w_i} \right)\\
			& \geq \sum\limits_{\substack{i=1}}^{s}  \left( y_{w_i} (1 - 1)  \right)   = 0.
		\end{align*}

		Similarly,
		\begin{align*}
			\mathcal{E}_{AB}&=\sum\limits_{\substack{i=s+1}}^{t}
			\left( x_{u'} y_{w_i} + z_{w_i} y_{u'}\right) -\sum_{u_j v_k \in E(A, B) }	\left(x_{u_j} y_{v_k} + x_{v_k} y_{u_j} \right) \\
			& = \sum\limits_{\substack{i=s+1}}^{t}\left( y_{w_i} + 0\right) - \sum\limits_{\substack{i=s+1}}^{t}\left(0 +x_{v_k} y_{w_i} \right) \\
			& \geq \sum\limits_{\substack{i=s+1}}^{t} \left( y_{w_i} (1 - 1)  \right)   = 0,
		\end{align*}
		and
		\begin{align*}
			\mathcal{E}_{B}&=\sum\limits_{\substack{i=t+1}}^{l} \left( x_{u'} y_{w_i} + z_{w_i} y_{u'}\right) - \sum_{v_j v_k \in E(B)}	\left(x_{v_j} y_{v_k} + x_{v_k} y_{v_j} \right)\\
			& = \sum\limits_{\substack{i=t+1}}^{l} \left( y_{w_i} + 0 - 0 - 0 \right) \geq  0.
		\end{align*}
		
		The assumption that there exists an edge in $ E(A)\cup E(A, B)\cup E(B) $ implies that $\mathcal{E}_A+\mathcal{E}_{AB}+\mathcal{E}_B>0$. By (\ref{eq1}), we deduce that $ \rho(G') - \rho(G) > 0$. Since $G'$ is a star with size $m$,  $\rho(G) < \rho(G') = \sqrt{m}$, which  contradicts  Lemma \ref{rhox}.
		So $u'$ is incident to every edge in $G$, i.e. $G$ is a star with size $m$. Then
		$\rho(G)= \sqrt{m}$, which also contradicts Lemma \ref{rhox}.
		Therefore, $ d_G(u') \leq 0.505m $.
	\end{proof}
	\begin{lemma}\label{xuu}
		There exists a vertex $ u'' \in L_1(u') \cup L_2(u') $ such that $ x_{u''} > 0.98 $.
	\end{lemma}
	
	\begin{proof}
		Recall that  $ x_{u'} = \max_{v \in V(G)} x_v = 1 $. Substituting $ u = u' $ and $ \varepsilon =  \frac{1}{1000}$ into (\ref{eq10}), we obtain
		\begin{align*}
			m < d_G(u') + \sum_{v \in \overline{L_1}(u')} \sum_{w \in L_1(u') \cup L_2(u')} x_w + 0.003m.
		\end{align*}
		Furthermore, by Lemma \ref{3.6}, we have
		\begin{align}
			\sum_{v \in \overline{L_1}(u')}\sum_{w \in L_1(u') \cup L_2(u')} x_w
			& > m - 0.003m - d_G(u')\nonumber\\
			& \geq m - 0.003m - 0.505m\nonumber\\
			&= 0.492m. \label{eq13}
		\end{align}
		
		Since $ u' \in L $, by Lemma \ref{du}, it follows that $ d_G(u') > 0.498m $. Define $ N_{L_1}(u') = N_G(u') \cap L_1(u') $ and let $ d_{L_1(u')}(u') = |N_{L_1(u')}(u')| $. Then, by Lemma \ref{l1},
		\begin{align*}
			d_{L_1(u')}(u') &\leq |L_1(u')| \leq |L|\\[2mm]
			& \leq 20000 < 0.00001m,
		\end{align*}
		where the last inequality holds as $ m $ is sufficiently large. Therefore,
		\begin{align*}
			d_{\overline{L_1}(u')}(u') = d_G(u') - d_{L_1(u')}(u') > 0.49799m.
		\end{align*}
		By combining this with (\ref{pmbds}), one  immediately obtains that
		\begin{align}
			\label{eq14}
			e(\overline{L_1}(u'), L_1(u') \cup L_2(u'))& \leq e(\overline{L_1}(u'), L) - d_{\overline{L_1}(u')}(u')\nonumber\\ &< m - 0.49799m \nonumber\\
			&= 0.50201m.
		\end{align}
		According to  (\ref{eq13}) and (\ref{eq14}), there exists a vertex $ u'' \in L_1(u') \cup L_2(u') $ such that
		\begin{align*}
			x_{u''} &\geq \frac{\sum_{v \in \overline{L_1}(u')} \sum_{w \in L_1(u') \cup L_2(u')} x_w}{e(\overline{L_1}(u'), L_1(u') \cup L_2(u'))}\\[2mm]
			&> \frac{0.492m}{0.50201m}
			> 0.98,
		\end{align*}
		as required.
	\end{proof}
	
	Set $ D = \{u', u''\} $, $ R = N_G(u') \cap N_G(u'')=\{u_1, \ldots, u_{|R|}\} $, and $ R_1 = V(G) \setminus (D \cup R)=\{v_1, \ldots, v_{|R_1|}\} $. 
	Next we will prove that the entries of the eigenvector $ \boldsymbol{x} $ corresponding to vertices in $ R \cup R_1 $ are very small.
	
	\begin{lemma}\label{xu}
		For each $u \in R \cup R_1$, we have $d_R(u) \leq 2$ and $x_u < 0.04$.
	\end{lemma}
	\begin{proof}
		Note that for $u \in R$, we have $d_D(u) = 2$, and $d_D(u) \leq 1$ for  $u \in R_1$.
		Moreover, we assert that $d_R(u) \leq 2$ for any vertex $u \in R \cup R_1$. Otherwise, $G$ would contain a subgraph isomorphic to $K_{3, 3}$, which is impossible.

		Recall that $ x_{u'} = 1 $ and $ x_{u''} > 0.98 $ (according to Lemma \ref{xuu}). By Lemma \ref{du}, we obtain
		\begin{align}
			d_G(u') > 0.498m \quad \text{and} \quad d_G(u'') > 0.488m.\label{3.12}
		\end{align}
		Since the total number of edges in graph $ G $ is $ m $, then for each $ u \in R \cup R_1 $,  $ d_G(u) < 0.015m $.

		To show that $x_u < 0.04$ for any $u \in R \cup R_1$,
		according to the definition of $ L $, we only need to consider the case $ u \in L $. Indeed, by Lemma \ref{du}, we obtain
		\begin{align*}
			\left(\frac{x_{u}}{2} - 0.002\right)m < d_G(u) < 0.015m,
		\end{align*}
		which gives $ x_{u} < 0.04 $.
	\end{proof}
\noindent
\textbf{Proof of Theorem \ref{th3}.}  We proceed the proof by the following claims.
\begin{claim}
	$ E(R) = \emptyset $ and $ N_{G}(v) = \{u'\} $ for any
	$ v \in R_{1} $.
\end{claim}
\begin{proof}
	Assume, for contradiction, $ E(R)\not=\emptyset $ or there exists a vertex $v$ in $R_1$ such that  $ N_{G}(v) \not= \{u'\} $. We deduce that $E(R)\cup E(R_1)\cup E(R, R_1)\cup E(u'', R_1)\not=\emptyset$.
	We replace all the edges $ u''v$ in $ E(u'', R_1)$ with $u'v$.
	And we delete all edges in $E(R)\cup E(R_1)\cup E(R, R_1)$, and for each deleted edge, we add a new vertex $w_i$ and an edge $u'w_i$. The obtained  graph is denoted by $ G'$.
	Set $|E(R)|=s, |E(R)|+|E(R, R_1)|=t, |E(R)|+|E(R, R_1)|+|E(R_1)|=l$.
	Let $ W=\{w_1, w_2, \dots, w_l\} $ be  the set of added  vertices, where $w_i \ (1\leq i\leq s), w_i\ (s+1\leq i\leq t), w_i\ (t+1\leq i\leq l)$ are the added vertices corresponding to the deleted edges in $E(R), E(R, R_1)$ and $E(R_1)$ respectively.
	For the graph $ G' $, let $$
	\boldsymbol{y} = \begin{pmatrix}
		1, y_{u''}, y_{u_1}, \dots, y_{u_{|R|}}, y_{v_1}, \dots, y_{v_{|R_1|}}, y_{w_1}, \dots, y_{w_{l}}
	\end{pmatrix}^{\mathrm{T}}
	$$ be
	the non-negative eigenvector corresponding to the spectral radius $ \rho(G') $ with maximum entry $1$.
	Then $y_{u'}=1, y_{v_i}=0 \ (1\leq i\leq |R_1|)$ and $y_{w_i}=y_{w_1}\ (2\leq i\leq |W|)$. For any $u_i\in R, w_j\in W$,   by (\ref{q1}), we have
	
	\begin{align}
		y_{u_i}& = \frac{ x_{u'}+x_{u''}}{\rho(G')} \leq \frac{2}{\rho(G')}\label{equ12}\\
		\quad y_{w_j} &= \frac{1}{\rho(G')}	\label{equ16}.
	\end{align}
	
	Let $ G^* = G \cup I_{|W|} $. Then $ \rho(G^*) = \rho(G) $ and $ |V(G^*)| = |V(G')| $.
	Recall that
	$\boldsymbol{x} = (x_{1}, x_{2}, \cdots, x_{n})^{\mathrm{T}}$ is a positive eigenvector corresponding to $ \rho(G)$ and  $x_{u'} = 1$.
	Then $$\boldsymbol{z}=
	(1, x_{u''}, x_{u_1}, \dots, x_{u_{|R|}}, x_{v_1}, \dots, x_{v_{|R_1|}},  z_{w_1},  \dots, z_{w_l})^{\mathrm{T}}	,$$
	where $z_{w_i}=0 \ (i=1, \ldots, l)$,  is a  non-negative eigenvector corresponding to $\rho( G^*)$.
	By Lemma \ref{sxl}, we obtain
	\begin{align}\label{equ13}
		\rho(G') - \rho(G) & = \frac{2}{\boldsymbol{z}^{\mathrm{T}} \boldsymbol{y}} \left(\mathcal{E}_{R} + \mathcal{E}_{RR_1} + \mathcal{E}_{R_1}+ \mathcal{E}_{u''}\right),
	\end{align}
	where
	\begin{align*}
		\mathcal{E}_R&=\sum\limits_{\substack{i=1}}^{s} \left( x_{u'} y_{w_i} + z_{w_i} y_{u'}\right) - \sum_{ u_j u_k \in E(R)} \left(x_{u_j} y_{u_k} + x_{u_k} y_{u_j} \right),\\
		\mathcal{E}_{RR_1}&=\sum\limits_{\substack{i=s+1}}^{t} 	\left( x_{u'} y_{w_i} + z_{w_i} y_{u'}\right) -\sum_{u_j v_k \in E(R, R_1) }	\left(x_{u_j} y_{v_k} + x_{v_k} y_{u_j} \right),\\
		\mathcal{E}_{R_1}&=\sum\limits_{\substack{i=t+1}}^{l} \left( x_{u'} y_{w_i} + z_{w_i} y_{u'}\right) - \sum_{v_j v_k \in E(R_1)}	\left(x_{v_j} y_{v_k} + x_{v_k} y_{v_j} \right),\\
		\mathcal{E}_{u''}&=\sum\limits_{u''v\in E(u'', R_1)}^{k} \left( x_{u'} y_{v} + x_{v} y_{u'}- x_{u''} y_{v} - x_{v} y_{u''}\right).
	\end{align*}
	By Lemma \ref{xu}, (\ref{equ12}) and (\ref{equ16}), we have
	\begin{align*}
		\mathcal{E}_{R} &= \sum\limits_{\substack{i=1}}^{s} \left( x_{u'} y_{w_i} + z_{w_i} y_{u'}\right) - \sum_{ u_j u_k \in E(R)} \left(x_{u_j} y_{u_k} + x_{u_k} y_{u_j} \right)\\
		&\geq \sum\limits_{\substack{i=1}}^{s} \left(y_{w_i} +0\right) - \sum_{u_j u_k \in E(R)} \left(0.04 y_{u_k} + 0.04y_{u_j} \right)\\
		& \geq  \sum\limits_{\substack{i=1}}^{s} \left( \frac{1}{\rho(G')}(1 - 0.08 - 0.08) \right)\geq 0.
	\end{align*}
	Similarly,
	\begin{align*}
		\mathcal{E}_{RR_1} &= \sum\limits_{\substack{i=s+1}}^{t} \left( x_{u'} y_{w_i} + z_{w_i} y_{u'}\right) - \sum_{ u_j v_k \in E(R,R_1)} \left(x_{u_j} y_{v_k} + x_{v_k} y_{u_j} \right)\\
		&\geq \sum\limits_{\substack{i=s+1}}^{t} \left(y_{w_i} +0 - 0.04 y_{v_k} - 0.04y_{w_i} \right)\\
		& \geq  \sum\limits_{\substack{i=s+1}}^{t} \left( \frac{1}{\rho(G')}(1 - 0.08 - 0.04) \right)\geq 0,
	\end{align*}
	\begin{align*}
		\mathcal{E}_{R_1} &= \sum\limits_{\substack{i=t+1}}^{l} \left( x_{u'} y_{w_i} + z_{w_i} y_{u'}\right) - \sum_{ v_j v_k \in E(R_1)} \left(x_{v_j} y_{v_k} + x_{v_k} y_{v_j} \right)\\
		&\geq \sum\limits_{\substack{i=t+1}}^{l} \left(y_{w_i} +0\right) \geq 0,
	\end{align*}
	and
	\begin{align*}
		\mathcal{E}_{u''}&=\sum\limits_{u''v\in E(u'', R_1)}^{k} \left( x_{u'} y_{v} + x_{v} y_{u'}- x_{u''} y_{v} - x_{v} y_{u''}\right)\\
		&=\sum\limits_{u''v\in E(u'', R_1)}^{k} \left( y_{v}(1- x_{u''} ) +x_{v}( 1 - y_{u''})\right)\geq 0.
	\end{align*}
	The assumption that $E(R)\cup E(R_1)\cup E(R, R_1)\cup E(u'', R_1)\not=\emptyset$ implies that $\mathcal{E}_{R}+\mathcal{E}_{RR_1}+\mathcal{E}_{R_1}+\mathcal{E}_{u''}>0$. By (\ref{equ13}), we get $\rho(G') - \rho(G) > 0$.
	Let $G''$ be the graph obtained from $G'$ by deleting the isolated vertices. Obviously, $G''$ is a planar graph with size $m$ and $\rho(G'')=\rho(G')>\rho(G)$,
	which contradicts the maximality of $ \rho(G) $.  So $ E(R) = \emptyset $ and $ N_{G}(v) = \{u'\} $ for any
	$ v \in R_{1} $.
\end{proof}
\begin{claim}
	$ |R_1| \leq 1 $
\end{claim}
\begin{proof}
	Suppose, for the sake of contradiction, that $ |R_1| \geq 2 $. Let $ v_1 $ and $ v_2 $ be two vertices in $ R_1 $. Since $ v_1 $ and $ v_2 $ are only adjacent to $ u' $, by symmetry,  $ x_{v_1} = x_{v_2} $. Define $ G' = G - v_2u' + v_1u'' $. Obviously, $ G' $ is still a planar graph and $u'$ is contained in $N_{G'}(u)$ for any vertex $u$ in $N_{G'}(u'')$. Let $\boldsymbol{y}$ denote the non-negative eigenvector corresponding to the spectral radius of $ G' $  with maximum entry $1$.
	Recall that  $ \boldsymbol{x} $ is the positive eigenvector corresponding to the spectral radius of $ G $  with maximum entry $1$.
	Then $x_{u'}=y_{u'}=1$ and $y_{v_2}=0$. By (\ref{q1}), we have
	\begin{align}\label{equ14}
		\rho(G)x_{v_2}=1 \  \mbox{and}\  \rho(G')y_{v_1}=y_{u'}+y_{u''}
	\end{align}

	We first show that $ y_{u''} > 0.24 $. By (\ref{3.12}),
	$d_G(u'') > 0.488m, $
	which implies
	$d_{G'}(u'') > 0.488m. $
	For any vertex $u$ in $N_{G'}(u'')$,  we have
	\begin{align*}
		\sum_{v\in N_{G'}(u)}y_v\geq y_{u'}=1.
	\end{align*}
	Then
	\begin{align*}
		\rho^2(G')y_{u''}&=\rho(G')\left(\sum_{u\in N_{G'}(u'')}y_u\right)\\[2mm]
		& =\sum_{u\in N_{G'}(u'')}\sum_{v\in N_{G'}(u)}y_v\\[2mm]
		& \geq \sum_{u\in N_{G'}(u'')} 1\\[2mm]
		&\geq d_{G'}(u'')> 0.488m.
	\end{align*}
	Combining this with  (\ref{rhosx}), we get
	\begin{align*}
		y_{u''} > \frac{0.488m}{\rho^2(G')}\geq  \frac{0.488m}{2m} > 0.24.
	\end{align*}
	Moreover, by Lemma \ref{sxl}, Lemma \ref{rhox}, Lemma \ref{xuu},   (\ref{rhosx}) and (\ref{equ14}), we obtain
	\begin{align*}
		\rho(G') - \rho(G) &\geq \frac{2}{\boldsymbol{x}^{\mathrm{T}}\boldsymbol{y}} \left( x_{u''} y_{v_1} + x_{v_1} y_{u''} - x_{u'} y_{v_2} - x_{v_2} y_{u'} \right) \\[2mm]
		&> \frac{2}{\boldsymbol{x}^{\mathrm{T}}\boldsymbol{y}} \left( 0. 98 y_{v_1} +0. 24x_{v_1}  - 0 -x_{v_2}\right) \\[2mm]
		&> \frac{2}{\boldsymbol{x}^{\mathrm{T}}\boldsymbol{y}} \left( 0. 98 \cdot \frac{1 + 0.24}{\rho(G')} -\frac{0.76}{\rho(G)} \right) \\[2mm]
		&> \frac{2}{\boldsymbol{x}^{\mathrm{T}}\boldsymbol{y}} \left( 0.98 \cdot \frac{1.24}{\sqrt{2m}} - \frac{0.76}{\sqrt{m}} \right) > 0,
	\end{align*}
	which implies that $ \rho(G') > \rho(G) $,  contradicting the maximality of $ G $. Therefore, $ |R_1| \leq 1 $.
\end{proof}
\begin{claim}
	$u'u'' \in E(G)$.
\end{claim}
\begin{proof}
	If $u'u''\not\in E(G)$, then we select a vertex $v$ in  $N_G(u'')$.  Define $G' = G - u''v + u'u''$. Obviously, $G'$ remains planar.  Since  $x_v \leq x_{u'}$, by Lemma \ref{yb}, $\rho(G') > \rho(G)$, which  contradicts the maximality of $G$. Hence, $u'u'' \in E(G)$.
\end{proof}
Since the vertices in $ R_1 $ are only connected to $ u' $, $ E(R) = \emptyset $, $ u'u'' \in G $, and $ |R_1| \leq 1 $,   $ G  \cong G_m $.\qed

\section{Concluding Remarks}
In this paper, we consider the spectral extremal problems of outerplanar and planar graphs with $m$ edges. We determine the spectral extremal graphs of outerplanar graphs with $ m \geq 64 $, and  planar graphs when $ m $ is sufficiently large. However, when $ m $ is small, we find that the spectral extremal  graph over all planar graphs with size $m$ is not $ G_m $.  Based on our findings, we propose the following conjecture.

\begin{conjecture}
	The planar graph on $m$ edges with the maximum spectral radius is
	\[
	\begin{cases}
		G_m, & \text{for } m \geq 33 \text{ or } m = 31, \\
		H_m, & \text{for } m \leq 30 \text{ or } m = 32,
	\end{cases}
	\]
	where
	\[
	H_m =
	\begin{cases}
		K_2 \vee P_{\frac{m}{3}}, & \text{if } m \mod 3 = 0, \\
		K_1 \vee \left((K_1 \vee P_{\frac{m-1}{3}}) \cup K_1\right), & \text{if } m \mod 3 = 1, \\
		K_2 \vee \left(P_{\frac{m-2}{3}} \cup K_1\right), & \text{if } m \mod 3 = 2.
	\end{cases}
	\]
\end{conjecture}


\end{document}